\documentclass[a4paper,reqno]{amsart}
\usepackage[utf8]{inputenc}
\usepackage{amsmath,enumitem}
\usepackage{booktabs,url}
\usepackage{amssymb,yhmath,mathrsfs,wasysym}
\usepackage{amsthm}
\usepackage{graphicx}
\usepackage{subfigure}
\usepackage[subfigure]{ccaption}
\usepackage{multicol,multirow,caption}
\usepackage{tikz,pgf}
\usetikzlibrary{arrows,automata,shapes,decorations.pathmorphing,decorations.markings,fadings,patterns}

\captionnamefont{\footnotesize\bf}
\captiontitlefont{\footnotesize}
\renewcommand{\thesubfigure}{\thefigure} \makeatletter \renewcommand{\p@subfigure}{} \renewcommand{\@thesubfigure}{\bf Figure \thesubfigure:\hskip\subfiglabelskip} \makeatother

\theoremstyle{plain}
\newtheorem{theorem}{Theorem}

\newtheorem{corollary}{Corollary}
\newtheorem{lemma}{Lemma}

\theoremstyle{definition}

\newcommand{\NN}{\mathbb{N}}

\newcommand{\CC}[0]{\mathbb{C}}
\newcommand{\HH}{\mathbb{H}}
\newcommand{\ie}[0]{\textnormal{i}}

\newcommand{\de}[0]{\mathrel{\mathop:}=}

\newcommand{\dif}[1]{\textnormal{d}#1}

\title{Lehmer pairs and derivatives of Hardy's $Z$-function}
\author{Aleksander Simoni\v{c}}
\address{Faculty of Mathematics and Physics \\ University of Ljubljana \\ Jadranska 19 \\ 1000 Ljubljana \\ Slovenia}
\email{aleksander.simonic@student.fmf.uni-lj.si}
\begin{document}

\begin{abstract}
Occurrences of very close zeros of the Riemann zeta function are strongly connected with Lehmer pairs and with the Riemann Hypothesis. The aim of the present note is to derive a condition for a pair of consecutive simple zeros of the $\zeta$-function to be a Lehmer pair in terms of derivatives of Hardy's $Z$-function. Furthermore, we connect Newman's conjecture with stationary points of the $Z$-function, and present some numerical results.
\end{abstract}

\maketitle
\thispagestyle{empty}

\section{Introduction}

Hardy's $Z$-function is defined by:
\[
   Z(t) \de e^{\ie\vartheta(t)}\zeta\left(\frac{1}{2}+\ie t\right), \;\; \vartheta(t) \de \frac{1}{2\ie}\log{\frac{\Gamma\left(\frac{1}{4}+\ie\frac{t}{2}\right)}{\Gamma\left(\frac{1}{4}-\ie\frac{t}{2}\right)}}-\frac{\log\pi}{2}t.
\]
In 1956, D.~H.~Lehmer found a pair of very close zeros of the $Z$-function. Visually this means that the graph $Z(t)$ barely crosses the $t$-axis at these zeros, and this poses a threat to the validity of the Riemann Hypothesis, see \cite{Edw}. The occurrence of such pairs of close zeros is known as \emph{Lehmer's phenomenon} and pairs are called \emph{Lehmer pairs}. The paper \cite{Csordas} gives a precise meaning to the notion of Lehmer pairs and proves a result which is today's most promising method to find lower bounds $\lambda$ of the \emph{de Bruijn-Newman constant} $\Lambda$. It should be noted that the Riemann Hypothesis is equivalent to $\Lambda\leq0$ and an infinite number of Lehmer pairs implies $\Lambda\geq0$, which is known as \emph{Newman's conjecture}.

Let $\left\{\gamma_1,\gamma_2\right\}$ be a pair of two simple zeros of the $Z$-function and define
\begin{equation}
\label{eq:gbar}
\bar{g}_{\left\{\gamma_1,\gamma_2\right\}} \de \left(\gamma_1-\gamma_2\right)^2 \sum_{\gamma\notin\{\gamma_1,\gamma_2\}} \frac{1}{(\gamma-\gamma_1)^2}+\frac{1}{(\gamma-\gamma_2)^2}
\end{equation}
where $\gamma$ goes through all zeros of the $Z$-function. Assuming the Riemann Hypothesis, a pair $\left\{\gamma_-,\gamma_+\right\}$ of two consecutive simple zeros of the $Z$-function is said to be a Lehmer pair if $\bar{g}_{\left\{\gamma_-,\gamma_+\right\}}<4/5$. If this is the case, then
\[
   \lambda_{\left\{\gamma_-,\gamma_+\right\}} \de \frac{\left(\gamma_+-\gamma_-\right)^2}{2\bar{g}_{\left\{\gamma_-,\gamma_+\right\}}}\left(\left(1-\frac{5}{4}\bar{g}_{\left\{\gamma_-,\gamma_+\right\}}\right)^{\frac{4}{5}}-1\right)
   \leq \Lambda.
\]
Because the bound $4/5$ is the least restrictive possible that allows the proof of the above inequality to go through, it is not surprising that Lehmer pairs are not so rare; the first thousand zeros contain $48$ Lehmer pairs.

In order to produce an ``analytic'' definition of Lehmer pair, Stopple gives in \cite[Theorem 1]{Stopple} a more restrictive definition in terms of the first three derivatives of the Riemann xi-function $\Xi(t)$ at the pair's zeros. In this note we extend his result to derivatives of the $Z$-function.

\begin{theorem}
\label{thm:main}
Let $Z(t)$ be Hardy's $Z$-function and define the real function $\widehat{F}$ by
\begin{equation}
\label{eq:F}
\widehat{F}(t) \de -\frac{Z'''}{Z'}(t)+\frac{3}{4}\left(\frac{Z''}{Z'}\right)^2(t).
\end{equation}
Let $\left\{\gamma_1,\gamma_2\right\}$ be a pair of two simple zeros of the $Z$-function and define
\[
   \hat{g}_{\left\{\gamma_1,\gamma_2\right\}} \de \frac{1}{3}\left(\gamma_1-\gamma_2\right)^2\left(\widehat{F}(\gamma_1) + \widehat{F}(\gamma_2)\right) - 2.
\]
Under the Riemann Hypothesis we have
\begin{equation}
\label{eq:leh.cond}
0 < \bar{g}_{\left\{\gamma_1,\gamma_2\right\}} - \hat{g}_{\left\{\gamma_1,\gamma_2\right\}} < 3\left(\gamma_1-\gamma_2\right)^2\left(\frac{1}{\gamma_1^2} + \frac{1}{\gamma_2^2}\right) < 3\hat{g}_{\left\{\gamma_1,\gamma_2\right\}}.
\end{equation}
\end{theorem}

This estimate is obviously very good for consecutive and relatively large zeros. The following immediate corollary gives conditions for a pair of zeros to be a Lehmer pair.

\begin{corollary}
\label{cor:main}
With notations and assumptions as in Theorem \ref{thm:main}, the value $\hat{g}_{\left\{\gamma_1,\gamma_2\right\}}$ is always strictly positive. If $\left\{\gamma_-,\gamma_+\right\}$ is a Lehmer pair, then $\hat{g}_{\left\{\gamma_-,\gamma_+\right\}}<4/5$. If
\begin{equation}
\label{eq:leh}
\hat{g}_{\left\{\gamma_-,\gamma_+\right\}} < \frac{4}{5} - 3\left(\gamma_+-\gamma_-\right)^2\left(\frac{1}{\gamma_-^2} + \frac{1}{\gamma_+^2}\right),
\end{equation}
then $\left\{\gamma_-,\gamma_+\right\}$ is a Lehmer pair.
\end{corollary}

Observe that the expression on the right hand side of \eqref{eq:leh} is very close to $4/5$ for large zeros, especially when we have a good candidate for a Lehmer pair. Therefore, it is reasonable to conjecture that we can omit this small term without changing the conclusion of Corollary \ref{cor:main}. We also conjecture that $\inf\left\{\hat{g}_{\left\{\gamma_-,\gamma_+\right\}}\right\}=0$.

In Section \ref{sec:stat} we briefly analyse a stronger condition for pairs of zeros to be a Lehmer pair. We take a similar approach to that of Stopple in the second part of his paper, but, in our case, working with stationary points of the $Z$-function instead of the $\zeta$-function, see Theorem \ref{thm:lambda}.

The main reason for considering such a rephrasing of the definition of Lehmer pair was a na\"{\i}ve question by the present author, if is it possible to prove that a given pair is Lehmer by simply computing derivatives. Unfortunately, values of the $\Xi$-function are very small, even for small $t$, and therefore inappropriate for numerical calculations. In Section \ref{sec:num} we demonstrate that it is possible to calculate derivatives of the $Z$-function by methods of numerical integration. Therefore, it is an interesting question to see whether Corollary \ref{cor:main} can be practical as a means of testing for Lehmer pairs.

\section{Some lemmas}

The following lemma is crucial in Stopple's approach and also in ours. Its proof is very simple, but for the convenience of the reader we include it here.

\begin{lemma}
\label{lemma:sch}
Assume that $f(s)$ and $g(s)$ are holomorphic functions on some domain $\Omega\subseteq\CC$ such that $f(s)=(s-z)g(s)$ and $g(z)\neq0$ for some $z\in\Omega$. Then
\[
   \left(\left(\frac{f''}{f'}\right)'+\frac{1}{4}\left(\frac{f''}{f'}\right)^2\right)(z) = 3\left(\frac{g'}{g}\right)'(z).
\]
\end{lemma}

\begin{proof}
We have $f'(s)=g(s)+(s-z)g'(s)$, $f''(s)=2g'(s)+(s-z)g''(s)$ and $f'''(s)=3g''(s)+(s-z)g'''(s)$. Therefore
\[
   \left(\frac{f''}{f'}\right)'(z) = \frac{f'''}{f'}(z) - \left(\frac{f''}{f'}\right)^2(z) = 3\frac{g''}{g}(z) - 4\left(\frac{g'}{g}\right)^2(z)
\]
and we get the desired equality.
\end{proof}

\begin{lemma}
\label{lemma:H}
Let $\zeta(s,z)=\sum_{n=0}^\infty(z+n)^{-s}$ be the Hurwitz zeta-function. For $t>1$ we have
\begin{equation}
\label{eq:lemmaH}
\frac{-56t^2}{\left(4t^2+1\right)^2} < \Re\left\{\zeta\left(2,\frac{1}{4}+\ie\frac{t}{2}\right)\right\} < \frac{3}{2t^2}.
\end{equation}
\end{lemma}

\begin{proof}
Choose an arbitrary $t>1$. It is easy to see that $\Re\left\{\zeta\left(2,1/4+\ie t/2\right)\right\} = \sum_{n=0}^\infty f(n)$ where
\[
f(n) \de 16\left(\left(1+4n\right)^2-4t^2\right)\left(\left(1+4n\right)^2+4t^2\right)^{-2}.
\]
We are interested in the behavior of this function on the interval $[0,\infty)$; at $n=0$ it is negative, on the interval $[0,\left(2\sqrt{3}t-1\right)/4]$ it increases  to positive maximum $1/2t^{-2}$, then it decreases to zero at infinity. Because of this behavior we have
\[
  \sum_{n=0}^\infty f(n) > f(0) + \int_{0}^\infty f(n)\dif{n} - \frac{1}{2t^2}
  = -\frac{\left(4t^2-1\right)\left(28t^2-1\right)}{2t^2\left(4t^2+1\right)^2}
\]
from which the left side of \eqref{eq:lemmaH} follows. On the other hand,
\[
   \sum_{n=0}^\infty f(n) < \int_{0}^\infty f(n)\dif{n} + \frac{1}{2t^2} < \frac{3}{2t^2}
\]
and the proof is complete.
\end{proof}

\begin{lemma}
\label{lemma:zeros}
Let $0<\gamma_1'\leq\gamma_2'\leq\gamma_3'\leq\ldots$ denote ordinates of zeros of the $\zeta$-function in the upper half-plane. Then $\gamma_{n+1}'-\gamma_n'<7$ for every $n\in\NN$.
\end{lemma}

\begin{proof}
Denote by $N(T)$ the number of $\gamma_n'$'s not exceeding $T$. One form of the Riemann-von Mangoldt formula is
\begin{equation}
\label{eq:rm}
   \left| N(T) - \frac{T}{2\pi}\log{\frac{T}{2\pi e}} - \frac{7}{8} \right| \leq a\log{T} + b\log{\log{T}} + c + \frac{1}{5T}
\end{equation}
for $T\geq e$ and suitable positive constants $a,b,c$. The best known values are $a=0.110$, $b=0.290$ and $c=2.290$, see \cite{Platt}. From \eqref{eq:rm} we obtain
\begin{flalign}
\label{eq:rm.est}
N(T+H) - N(T) > &\left(\frac{H}{2\pi}-2a\right)\log{(T+H)} - 2b\log{\log{(T+H)}} \nonumber \\
&+ \frac{T}{2\pi}\log\left(1+\frac{H}{T}\right) - \frac{H}{2\pi}\log{(2\pi e)} - 2c - \frac{2}{5T}.
\end{flalign}
Therefore, for $H>4\pi a\approx1.38$ a constant $T_0$ exists such that $\gamma_{n+1}'-\gamma_n'\leq H$ for every $\gamma_n'>T_0$. The main advantage of \eqref{eq:rm.est} is that we can calculate $T_0$ for a given $H$.

Since $\gamma_2'-\gamma_1'<7$, it is enough to demonstrate
\begin{equation}
\label{eq:zeros.est2}
\gamma_{n+1}'-\gamma_n'<H=6\;\;\textrm{for}\;\;n\geq2.
\end{equation}
With help of \emph{Mathematica} we obtain from \eqref{eq:rm.est} the bound $T_0=35370$ (just for comparison, $H=1.4$ gives $T_0=4.7\times10^{1497}$). Furthermore, it is known (see \cite[p.~281]{Trud}) that for $T\in\left[0,6.8\times10^6\right]$ inequality \eqref{eq:rm} is true for constants $a=b=0$ and $c=2$. This fact lowers our bound to $T_0=412$. Up to this value we easily verified \eqref{eq:zeros.est2} by computer.
\end{proof}

\section{Proof of Theorem \ref{thm:main}}
\label{sec:proof}

Define the function $H$ by
\begin{equation}
\label{eq:H}
H(s) \de \frac{1}{2}(1-s)s\pi^{-s/2}\Gamma\left(\frac{s}{2}\right).
\end{equation}
Then the Riemann $\xi$-function is given by $\xi(s)\de H(s)\zeta(s)$ and $\Xi(t)=\xi(1/2+\ie t)$. Let $\rho_0$ be some simple zero of $\xi(s)$. By the Hadamard product formula for the $\xi$-function we have
\begin{equation}
\label{eq:xi}
\xi(s) = -\left(s-\rho_0\right)\prod_{\rho\neq\rho_0}e^{Bs+s/\rho_0}\rho_0^{-1}\left(1-\frac{s}{\rho}\right)e^{s/\rho}
\end{equation}
where $B$ is some constant. Applying Lemma \ref{lemma:sch} on \eqref{eq:xi}, we get
\begin{flalign}
F(\rho_0) \de \frac{1}{3}\left(\left(\frac{\xi''}{\xi'}\right)'+\frac{1}{4}\left(\frac{\xi''}{\xi'}\right)^2\right)(\rho_0) &= -\sum_{\rho\neq\rho_0}\left(\rho-\rho_0\right)^{-2} \label{eq:Fz} \\
& = \frac{1}{3}\left(\frac{\xi'''}{\xi'} - \frac{3}{4}\left(\frac{\xi''}{\xi'}\right)^2\right)(\rho_0) \nonumber.
\end{flalign}
If we take two simple zeros $\rho_1$ and $\rho_2$ and define $\Delta\de\rho_1-\rho_2$, then \eqref{eq:Fz} gives us
\begin{equation}
\label{eq:sum}
\Delta^2 \sum_{\rho\notin\left\{\rho_1,\rho_2\right\}}\frac{1}{\left(\rho-\rho_1\right)^{2}}+\frac{1}{\left(\rho-\rho_2\right)^{2}} = -\Delta^2\left(F(\rho_1)+F(\rho_2)\right)-2
\end{equation}
The left hand side of \eqref{eq:sum} is very similar to expression \eqref{eq:gbar}. Indeed, if we assume the Riemann Hypothesis, then $\bar{g}_{\{\gamma_1,\gamma_2\}}=\left(\gamma_1-\gamma_2\right)^2\left(F\left(\gamma_1\right)+F\left(\gamma_2\right)\right)-2$ where
\[
   F(t) = -\frac{1}{3}\left(\left(\frac{\Xi''}{\Xi'}\right)'+\frac{1}{4}\left(\frac{\Xi''}{\Xi'}\right)^2\right)(t).
\]
We would like to express function $F$ firstly in terms of the $\zeta$-function and then in terms of the $Z$-function. Taking $\zeta(\rho_0)=0$ into account, we simply differentiate $\xi(s)=H(s)\zeta(s)$ to find that
\begin{gather*}
\frac{\xi'''}{\xi'}(\rho_0) = 3\frac{H''}{H}(\rho_0) + 3\frac{H'\zeta''}{H\zeta'}(\rho_0) + \frac{\zeta'''}{\zeta'}(\rho_0), \\
\frac{3}{4}\left(\frac{\xi''}{\xi'}\right)^2(\rho_0) = 3\left(\frac{H'}{H}\right)^2(\rho_0) + 3\frac{H'\zeta''}{H\zeta'}(\rho_0) + \frac{3}{4}\left(\frac{\zeta''}{\zeta'}\right)^2(\rho_0)
\end{gather*}
and finally
\begin{equation}
\label{eq:funcF}
F(\rho_0) = \left(\frac{H'}{H}\right)'(\rho_0) + \frac{1}{3}\widetilde{F}(\rho_0) \;\;
\textrm{where} \;\; \widetilde{F}(\rho_0) \de \left(\frac{\zeta'''}{\zeta'} - \frac{3}{4}\left(\frac{\zeta''}{\zeta'}\right)^2\right)(\rho_0).
\end{equation}
By \eqref{eq:H} we obtain
\begin{equation}
\label{eq:funcH}
\left(\frac{H'}{H}\right)'(\rho_0) = \left(\log{H}\right)''(\rho_0) = -\frac{1}{(1-\rho_0)^2} - \frac{1}{\rho_0^2} + \frac{1}{4}\zeta\left(2,\frac{\rho_0}{2}\right).
\end{equation}

Let $\rho_0=1/2+\ie\gamma_0$. Straightforward calculations with $Z(\gamma_0)=0$ in mind give us
\begin{gather*}
\zeta'(\rho_0) = -\ie e^{-\ie\vartheta(\gamma_0)}Z'(\gamma_0), \\
\zeta''(\rho_0) = e^{-\ie\vartheta(\gamma_0)}\left(2\ie\vartheta'(\gamma_0)Z'(\gamma_0)-Z''(\gamma_0)\right), \\
\zeta'''(\rho_0) = e^{-\ie\vartheta(\gamma_0)}\left(\left(3\vartheta''(\gamma_0)-3\ie\vartheta'^2(\gamma_0)\right)Z'(\gamma_0)+3\vartheta'(\gamma_0)Z''(\gamma_0)+\ie Z'''(\gamma_0)\right).
\end{gather*}
From this, \eqref{eq:funcF} and \eqref{eq:F} we obtain $\widetilde{F}(\rho_0) = \widehat{F}(\gamma_0) + \ie 3\vartheta''(\gamma_0)$ where this is indeed the decomposition into the real and the imaginary part of $\widetilde{F}(\rho_0)$. By \eqref{eq:funcH} we have
\[
   \left(\frac{H'}{H}\right)'(\rho_0) = \frac{2\gamma_0^2-1/2}{\left(\gamma_0^2+1/4\right)^2} + \frac{1}{4}\zeta\left(2,\frac{1}{4}+\ie\frac{\gamma_0}{2}\right).
\]
Since $\Im\left\{\zeta\left(2,\rho_0/2\right)\right\} = -4\vartheta''(\gamma_0)$ by the definition of the Hurwitz zeta-function and the function $\vartheta$, $F(\rho_0) = (1/3)\widehat{F}(\gamma_0) + \varepsilon(\gamma_0)$
where
\[
   \varepsilon(\gamma_0) \de \frac{2\gamma_0^2-1/2}{\left(\gamma_0^2+1/4\right)^2} + \frac{1}{4}\Re\left\{\zeta\left(2,\frac{1}{4}+\ie\frac{\gamma_0}{2}\right)\right\}.
\]
By Lemma \ref{lemma:H} we have $0<\varepsilon(\gamma_0)<3\gamma_0^{-2}$, and by \eqref{eq:sum} we have
\[
   \bar{g}_{\left\{\gamma_1,\gamma_2\right\}}=\hat{g}_{\left\{\gamma_1,\gamma_2\right\}}+\left(\gamma_1-\gamma_2\right)^2\left(\varepsilon(\gamma_1)+\varepsilon(\gamma_2)\right).
\]
The combination of both gives the first two inequalities in \eqref{eq:leh.cond}.

To obtain the third inequality we need Lemma \ref{lemma:zeros}. Since $\gamma_2'-\gamma_1'<7<\gamma_1'/2$ by concrete calculations, it follows $\gamma_{n+1}'-\gamma_n'<\gamma_n'/2$ for every $n\in\NN$. Assuming the Riemann Hypothesis, by \eqref{eq:gbar} this implies $\bar{g}_{\left\{\gamma_1,\gamma_2\right\}}>4\left(\gamma_1-\gamma_2\right)^2\left(\gamma_1^{-2}+\gamma_2^{-2}\right)$ for every distinct zeros $\gamma_1,\gamma_2\geq\gamma_1'$ of the $Z$-function. This completes the proof of Theorem \ref{thm:main}.

\section{Stationary points of the $Z$-function}
\label{sec:stat}

Writing $Pf'(t)\de\left(f''/f'\right)'(t)$ for the \emph{pre-Schwarzian derivative} of $f'$, the assertion $-\left(\gamma_--\gamma_+\right)^2\left(P\Xi'(\gamma_-)+P\Xi'(\gamma_+)\right)<42/5$ implies that $\{\gamma_-,\gamma_+\}$ is a Lehmer pair. Stopple named such a pair \emph{a strong Lehmer pair} and showed by concrete calculations that this is indeed a much stronger condition. Define
\[
   \hat{\hat{g}}_{\left\{\gamma_-,\gamma_+\right\}} \de \frac{1}{3}\left(\gamma_{-}-\gamma_+\right)^2\left(-PZ'\left(\gamma_-\right)-PZ'\left(\gamma_+\right)\right)-2.
\]
By Corollary \ref{cor:main}, if
\begin{equation}
\label{eq:err1}
\hat{\hat{g}}_{\left\{\gamma_-,\gamma_+\right\}} < \frac{4}{5} - 3\left(\gamma_+-\gamma_-\right)^2\left(\frac{1}{\gamma_-^2} + \frac{1}{\gamma_+^2}\right),
\end{equation}
then $\left\{\gamma_-,\gamma_+\right\}$ is a Lehmer pair.

Stopple obtained in \cite[Theorem 3]{Stopple} a representation of $P\Xi'(\gamma)$ in terms of nearby zeros of $\zeta'(s)$. We can derive a much simpler expression for $PZ'(\gamma)$ of a similar nature. The $Z$-function is on the right half-plane $\HH_R\de\{z\in\CC\colon \Re\{z\}>0\}$ a holomorphic function. R.~Hall proved in \cite[Theorem 2]{Hall} that, assuming the Riemann Hypothesis, all zeros $u$ of $Z'(z)$ in $\HH_R$ are real. Furthermore, Lehmer's theorem \cite[\S8.3]{Edw} asserts that such a point $u$ is unique between consecutive zeros of the $Z$-function. We denote this nondecreasing sequence by $\{u_n\}$. For $t>e$ define the interval $\mathcal{I}_t\de[t-1/\log{\log{t}},t+1/\log{\log{t}}]$. We further assume the hypothesis
\begin{equation}
\label{eq:hypothesis}
\frac{1}{\zeta'\left(1/2+\ie\gamma\right)} = O(\gamma)
\end{equation}
for all positive zeros $\gamma$ of the $Z$-function. This follows from the ``weak'' Mertens' conjecture
\begin{equation*}
\int_{1}^{T}\left(\frac{M(x)}{x}\right)^2\dif{x} = O(\log{T})
\end{equation*}
where $M(x)=\sum_{n\leq x}\mu(n)$ and $\mu$ is the M\"{o}bius function, see Section 14.29 and proof of 14.30 in \cite{Titch}. Observe that \eqref{eq:hypothesis} implies the simplicity of zeros. Assuming all this, with $r=1/\log{\log{\gamma}}$ and $s_0=1/2+\ie\gamma$ a similar procedure as in the proof of \cite[Lemma 2]{Stopple} gives us
\begin{equation}
\label{eq:estimation}
-PZ'(\gamma) = \sum_{u\in\mathcal{I}_\gamma} \frac{1}{\left(\gamma-u\right)^2} + O\left(\log{\gamma}\log^2{\log{\gamma}}\right).
\end{equation}
A straightforward interpretation of this estimate is contained in the following theorem.

\begin{theorem}
\label{thm:lambda}
Assume the Riemann Hypothesis and \eqref{eq:hypothesis}. Let $\gamma_n$ be the $n$-th zero and $u_n\in\left(\gamma_n,\gamma_{n+1}\right)$ the $n$-th stationary point of $Z(t)$. Define
\[
   \alpha_n \de \frac{u_n-\gamma_n}{\gamma_{n+1}-\gamma_n}.
\]
If the set
\[
\mathscr{Z}\de\left\{n\geq2\colon \gamma_{n+1}-\gamma_n<\frac{1}{\log{\gamma_n}},\left\{u_{n-1},u_{n+1}\right\}\not\subset\mathcal{I}_{\gamma_n}\cup\mathcal{I}_{\gamma_{n+1}},
\alpha_n \in [0.44,0.56]\right\}
\]
is infinite, then the de Bruijn-Newman constant $\Lambda$ is $0$.
\end{theorem}

\begin{proof}
From \eqref{eq:estimation} it follows that
\begin{flalign}
\hat{\hat{g}}_{\left\{\gamma_n,\gamma_{n+1}\right\}} &= \frac{1}{3}\sum_{u\in\mathcal{I}_{\gamma_n}} \left(\frac{\gamma_{n+1}-\gamma_n}{u-\gamma_{n}}\right)^2 + \frac{1}{3}\sum_{u\in\mathcal{I}_{\gamma_{n+1}}} \left(\frac{\gamma_{n+1}-\gamma_n}{\gamma_{n+1}-u}\right)^2 - 2 \nonumber \\
&+ \left(\gamma_{n+1}-\gamma_n\right)^2 O\left(\log{\gamma_{n+1}}\log^2{\log{\gamma_{n+1}}}\right). \label{eq:err2}
\end{flalign}
The first assertion from the set $\mathscr{Z}$ guarantees that the error term tends to zero at infinity. Thus, for $\varepsilon=10^{-3}$ there exists an infinite subset $\mathscr{Z}'\subset\mathscr{Z}$ such that for every $n\in\mathscr{Z}'$ the absolute values of the error terms in \eqref{eq:err1} and \eqref{eq:err2} are not greater than $\varepsilon$. By the second assertion, the only contributing stationary point in the above sums is $u_n$. By the third assertion, for $n\in\mathscr{Z}'$ we have
\[
   \hat{\hat{g}}_{\left\{\gamma_n,\gamma_{n+1}\right\}} < \frac{1}{3}\left(\frac{1}{\alpha_n^2}+\frac{1}{(1-\alpha_n)^2}\right)-2+\varepsilon < 0.79 + \varepsilon < \frac{4}{5}-\varepsilon.
\]
Therefore, there exist an infinite number of Lehmer pairs, which implies $\Lambda=0$.
\end{proof}

Famous Montgomery's pair correlation conjecture states that
\[
   \liminf_{n\to\infty}\left(\gamma_{n+1}-\gamma_n\right)\log{\gamma_n} = 0,
\]
see \cite{BuiMilinovichNg} for a brief overview of the problem. This implies that the first assertion from the set $\mathscr{Z}$ is true for an infinite number of pairs. The author has calculated that the first two million zeros include $4637$ pairs of zeros which satisfy the first assertion, while $1901$ pairs actually belong to the set $\mathscr{Z}$.

\section{Numerical results}
\label{sec:num}

We can use Cauchy's integral formula to express higher derivatives since the $Z$-function is holomorphic in a neighborhoods of its zeros. This gives
\begin{equation}
\label{eq:int}
Z^{(n)}(x) = \frac{n!}{r^n}\int_{0}^{1}e^{-2\pi n\ie t}Z\left(x+re^{2\pi\ie t}\right)\dif{t}
   = \frac{n!}{r^n}\int_{0}^{1} G(t)\dif{t}
\end{equation}
where
\[
   G(t) \de \Re\left\{Z\left(x+re^{2\pi\ie t}\right)\right\}\cos{(2\pi nt)}+\Im\left\{Z\left(x+re^{2\pi\ie t}\right)\right\}\sin{(2\pi nt)}
\]
is a real function. We used \emph{Mathematica}'s function \texttt{RiemannSiegelZ} for the approximate calculation of the integral \eqref{eq:int} by the composite trapezoidal rule with tolerance $10^{-7}$. Observe that the parameter $n$ is involved only in the trigonometic functions which enables the simultaneous calculation of derivatives. The results we get for three pairs are summarized in Table \ref{tab:der}.

\begin{table}[h]
\begin{tabular}{lrrrr}
\hline \noalign{\smallskip}
$n$ & $\gamma_n$ & $Z'\left(\gamma_n\right)$ & $Z''\left(\gamma_n\right)$ & $Z'''\left(\gamma_n\right)$ \\
\hline \noalign{\smallskip}
$34$ & $111.02953554$ & $-1.590846$ & $\phantom{-}3.8401$ & $\phantom{-}1.4834$ \\
$35$ & $111.87465918$ & $\phantom{-}1.361151$ & $\phantom{-}2.2575$ & $-4.6657$ \\
$6709$ & $7005.06286617$ & $\phantom{-}0.414558$ & $-21.3028$ & $-57.1590$ \\
$6710$ & $7005.10056467$ & $-0.427037$ & $-23.2882$ & $-47.9737$ \\
$1048449114$ & $388858886.00228512$ & $-0.008173$ & $\phantom{-}150.552$ & $\phantom{-}123.981$ \\
$1048449115$ & $388858886.00239369$ & $\phantom{-}0.008173$ & $\phantom{-}150.565$ & $\phantom{-}122.665$ \\
\hline
\end{tabular}
\caption{Values of derivatives of the $Z$-function for three pairs of zeros.}
\label{tab:der}
\end{table}

The first pair is actually the first Lehmer pair while the second pair is probably the most famous one because Lehmer found it. The third pair is used in \cite{Csordas93} in order to obtain a lower bound for the de Bruijn-Newman constant.

These numbers are in agreement with calculations derived by the Python's \emph{mpmath} function \texttt{siegelz(z,derivative=n)}. This function gives for $1\leq n\leq4$ the $n$-th derivative of the $Z$-function at $z$.

\begin{table}[h]
\begin{tabular}{lrrr}
\hline \noalign{\smallskip}
$n$ & $-\left(\gamma_n-\gamma_{n+1}\right)^2PZ'\left(\gamma_n\right)$ & $\hat{g}$ & $\hat{\hat{g}}$ \\
\hline \noalign{\smallskip}
$34$ & 4.83 & \multirow{2}{*}{$0.57$} & \multirow{2}{*}{$1.08$} \\
$35$ & 4.41 &  &  \\
$6709$ & 3.95 & \multirow{2}{*}{$7\times10^{-3}$} & \multirow{2}{*}{$0.67$} \\
$6710$ & 4.07 &  &  \\
$1048449114$ & 4.00 & \multirow{2}{*}{$3\times10^{-5}$} & \multirow{2}{*}{$0.67$} \\
$1048449115$ & 4.00 &  &  \\
\hline
\end{tabular}
\caption{Values of $\hat{g}$ and $\hat{\hat{g}}$ for three pairs of zeros.}
\label{tab:g}
\end{table}

We use the values in Table \ref{tab:der} to obtain $\hat{g}$ and $\hat{\hat{g}}$ for corresponding pairs of zeros. The results are displayed in Table \ref{tab:g}. Observe that the value $\hat{g}$ for the second and the third pair is very close to zero while $\hat{\hat{g}}$ is approximately $2/3$. By Corollary \ref{cor:main} we know that $\hat{g}$ must be positive. Therefore, we have further evidence of ``the near failure of the Riemann Hypothesis''. On the other hand, by Theorem \ref{thm:lambda} the value $2/3$ is ``optimal'' for $\hat{\hat{g}}$ in the sense that the stationary point is exactly in the middle of the pair's zeros ($\alpha_n=1/2$) and the error term is small enough to be neglected. This is expected to be true for extremely close zeros.

\bigskip
\noindent\textbf{Acknowledgements.} We would like to thank Jeffrey Stopple for his interest in the subject, and Timothy Trudgian for his helpful remarks.

%%%%%%%%%%%%%%%%%%%%%%%%%%%%%%%%%%%%%%%%%%%%%%%%%%%%%%%%%%%%%%%%%%%%%%%%%%%%%%%%%%%%%%%%%%

\providecommand{\bysame}{\leavevmode\hbox to3em{\hrulefill}\thinspace}
\providecommand{\MR}{\relax\ifhmode\unskip\space\fi MR }
% \MRhref is called by the amsart/book/proc definition of \MR.
\providecommand{\MRhref}[2]{%
  \href{http://www.ams.org/mathscinet-getitem?mr=#1}{#2}
}
\providecommand{\href}[2]{#2}

%\bibliography{C:/Users/Alex/Documents/knjiznica/knjige,C:/Users/Alex/Documents/knjiznica/clanki,C:/Users/Alex/Documents/knjiznica/razno}
%\bibliographystyle{amsalpha}

\end{document}